\documentclass[a4paper, reqno]{amsart}

\usepackage{amsfonts, amsmath,  amssymb, bm, color, enumerate, esint, mathrsfs, xcolor}  
\usepackage{mathabx}
\usepackage{graphicx}

\newcommand{\C}{\mathbb{C}}

\newcommand{\R}{\mathbb{R}}
\newcommand{\Z}{\mathbb{Z}}
\newcommand{\N}{\mathbb{N}}
\newcommand{\Q}{\mathbb{Q}}
\newcommand{\ud}{\mathrm{d}}

\newcommand{\cC}{\mathcal{C}}
\newcommand{\cA}{\mathcal{A}}
\newcommand{\cB}{\mathcal{B}}

\newcommand{\fo}{\mathfrak{o}}
\newcommand{\fm}{\mathfrak{m}}

\newcommand{\bx}{\mathbf{x}}
\newcommand{\by}{\mathbf{y}}
\newcommand{\bxi}{\bm{\xi}}

\newcommand{\rad}{\mathrm{rad}\,}

\theoremstyle{plain}
\newtheorem{theorem}{Theorem}[section]
\newtheorem{lemma}[theorem]{Lemma}
\newtheorem{proposition}[theorem]{Proposition}
\newtheorem{corollary}[theorem]{Corollary}

\theoremstyle{definition}
\newtheorem{definition}[theorem]{Definition}

\newtheorem{remark}[theorem]{Remark}

\title[Non-archimedean Littlewood--Paley theory for curves]{A non-archimedean variant of Littlewood--Paley theory for curves}

\author[J. Hickman]{ Jonathan Hickman }
\address{School of Mathematics and Maxwell Institute for Mathematical Sciences, University of Edinburgh, James Clerk Maxwell Building, Kings Buildings, Peter Guthrie Tait Road, Edinburgh, EH9 3FD, UK.}
\email{jonathan.hickman@ed.ac.uk}

\author[J. Wright]{ James Wright }
\address{School of Mathematics and Maxwell Institute for Mathematical Sciences, University of Edinburgh, James Clerk Maxwell Building, Kings Buildings, Peter Guthrie Tait Road, Edinburgh, EH9 3FD, UK.}
\email{j.r.wright@ed.ac.uk}
\date{}

\begin{document}

\maketitle

\begin{abstract}
    We prove a variant of a square function estimate for the extension operator associated to the moment curve
    in non-archimedean local fields. The arguments rely on a structural analysis of congruences (sublevel sets) of univariate polynomials over field extensions of the base field. Our analysis can be adapted to the archimedean setting as well.
\end{abstract}



\section{Introduction}




\subsection{Statement of the results} This paper concerns the Fourier restriction theory for curves and associated Littlewood--Paley-type inequalities. Classically, this theory forms part of \textit{Euclidean} harmonic analysis, however here we explore these questions in the setting of a general locally compact topological field $K$ with a nontrivial topology. Such fields carry a natural absolute value $|\,\cdot\,|_K$ and a Haar measure $\mu$. They are classified as archimedean local fields (when $K = {\mathbb R}$ is the real field or when $K = {\mathbb C}$ is the complex field) or non-archimedean local fields such as the $p$-adic field ${\mathbb Q}_p$. The Littlewood--Paley theory for curves is well known when $K = {\mathbb R}$ is the real field so we will state and prove our results for non-archimedean local fields. In an appendix we will show how to adapt our arguments to work in the archimedean setting.  

Let $(K, |\,\cdot\,|_K)$ be a non-archimedean local field with ring of integers $\fo_K$, residue class field $k_K$, uniformiser $\pi_K$ and $q_K := |\pi|_K^{-1}$. For the reader's convenience, we will review some of the basic concepts of analysis over local fields in  \S\ref{sec: background} below. Fix an additive character $e \colon K \to \C$ such that $e$ restricts to the constant function 1 on $\fo_K$ and to a non-principal character on $\pi_K^{-1}\fo_K$. For $n \geq 2$, we define the \textit{extension operator} associated to the moment curve by
\begin{equation}\label{eq: padic ex op}
    Ef(\bx) := \int_{\fo_K} e(x_1t + x_2t^2 + \cdots + x_n t^n) f(t)\,\ud \mu(t) \qquad \textrm{for all $f \in L^1(K)$ and $\bx \in K^n$.}
\end{equation}
Here and below, integration is taken with respect to the Haar measure $\mu$ on $K$, which is normalised so that $\mu(\fo_K) = 1$. 

The operator $E$ is a fundamental object of study in the \textit{Fourier restriction theory} over local fields $K$. This theory was investigated systematically by the authors in \cite{HW2018}, with a focus on the problem of determining Lebesgue space mapping properties. Here we are interested in \textit{Littlewood--Paley} or \textit{square function} inequalities for the operator \eqref{eq: padic ex op}. To describe the setup, fix $\alpha \in \N$ and let $\mathcal{I}(q_K^{-\alpha})$ denote the collection of $q_K^{\alpha}$ distinct balls of the form
\begin{equation*}
    B_K(x; q_K^{-\alpha}) := \{ t \in \fo_K : |t-x|_K \leq q_K^{-\alpha} \}, \qquad x \in \fo_K.
\end{equation*}
Thus, $\mathcal{I}(q_K^{-\alpha})$ defines a decomposition of $\fo_K$, which induces a decomposition of the extension operator
\begin{equation}\label{eq: ex dec}
    Ef = \sum_{I \in \mathcal{I}(q_K^{-\alpha})} E_If \qquad \textrm{where $E_If := E\big(f \chi_I\big)$ for all $I \in \mathcal{I}(q_K^{-\alpha})$.}
\end{equation}
Here $\chi_I$ denotes the characteristic function of $I \in \mathcal{I}(q_K^{-\alpha})$.

\begin{theorem}\label{thm: sf} Let $(K, |\,\cdot\,|_K)$ be a non-archimedean local field and $\mathrm{char}\,k_K > n \geq 2$. For all $1 \leq m \leq n$ and all $\alpha \in \N$, the identity
\begin{equation*}
    \|Ef\|_{L^{2m}(B(\bx, q_K^{\alpha n}))} \leq (m!)^{1/2m}  \Big\|\big( \sum_{I \in \mathcal{I}(q_K^{-\alpha})} |Ef_I|^2\big)^{1/2}\Big\|_{L^{2m}(B(\bx, q_K^{\alpha n}))}
\end{equation*}
holds for all $f \in L^1(\fo_K)$ and all $\bx \in K^n$.
\end{theorem}

Throughout the paper, $L^p$ norms are taken with respect to the Haar measure on $K^n$ given by the $n$-fold product of $\mu$ above. The balls $B(\bx, q_K^{\alpha n})$ are defined with respect to the $\ell^{\infty}$ norm induced by $|\,\cdot\,|_K$: see \S\ref{sec: background} for further details. 

Theorem~\ref{thm: sf} is an analogue of a Euclidean result from \cite{Prestini1984a, Prestini1984b, GGPRY2021}, as described below in \S\ref{subsec: Euclid}. Moreover, recently square function inequalities of this type were investigated in the local field setting in \cite{BBH} in the case $n=2$ for general polynomial curves.\footnote{The methods of \cite{BBH} imply bounds for $n \geq 2$ but only at the level of $m = 2$.} 

By a well-known \textit{$2n$-orthogonality argument} due to C\'ordoba and Fefferman, the proof of Theorem~\ref{thm: sf} reduces to establishing the following number-theoretic proposition. 

\begin{proposition}\label{prop: main} Let $(K, |\,\cdot\,|_K)$ be a non-archimedean local field, $\mathrm{char}\,k_K > n \geq 2$ and $a \in \N$. Suppose $(x_1, \dots, x_n)$, $(y_1, \dots,  y_n) \in (\fo_K)^n$ satisfy
\begin{equation}\label{system: powers}
   |x_1^j + \cdots + x_n^j - y_1^j - \cdots - y_n^j|_K \leq q_K^{-n a} \qquad \textrm{for $1 \leq j \leq n$.}
\end{equation}
Then there exists a permutation $\sigma$ on $\{1, \cdots, n\}$ such that $|x_j - y_{\sigma(j)}|_K \leq q_K^{-a}$ for all $1 \leq j \leq n$.
\end{proposition}

Proposition~\ref{prop: main} examines the structure of `almost solutions' to a Vinogradov-type system of equations. In particular, it can be roughly interpreted as saying that every `almost solution' to the system $x_1^j + \cdots + x_n^j = y_1^j + \cdots + y_n^j$ for $1 \leq j \leq n$ is `almost trivial'. A similar statement appears in the Euclidean context in \cite{GGPRY2021}, although the method of proof used in \cite{GGPRY2021} breaks down completely in the non-archimedean setting (see the discussion in \S\ref{subsec: proof} below). 




\subsection{Motivation: the Euclidean case}\label{subsec: Euclid} It is instructive to contrast Theorem~\ref{thm: sf} with counterpart results in the Euclidean setting. For $n \geq 2$ let $\gamma \colon [0,1] \to \R^n$ be a $C^n$ curve in $n$-dimensional Euclidean space which satisfies the non-degeneracy hypothesis $\det(\gamma'(t) \cdots \gamma^{(n)}(t)) \neq 0$ for all $t \in [0,1]$. Define the associated extension operator
\begin{equation*}
    Ef(\bx) := \int_0^1 e^{2 \pi i \bx \cdot \gamma(t)} f(t)\,\ud t \qquad \textrm{for all $f \in L^1([0,1])$ and $\bx \in \R^n$.}
\end{equation*}
Let $0 < \delta \leq 1$ be a dyadic number and $\mathcal{I}(\delta)$ be the decomposition of $[0,1]$ into closed dyadic intervals of length $\delta$. We decompose the extension operator as in \eqref{eq: ex dec}, with $q_K^{-\alpha}$ replaced by $\delta$. Under these hypotheses, it is known that for each $1 \leq m \leq n$ there exists a constant $C_m \geq 1$ such that
\begin{equation}\label{eq: euclid sf}
    \|Ef\|_{L^{2m}(B_{\delta^{-n}})} \leq C_m\Big\|\big( \sum_{I \in \mathcal{I}(\delta)} |Ef_I|^2\big)^{1/2}\Big\|_{L^{2m}(w_{B_{\delta^{-n}}})}
\end{equation}
holds for all $f \in L^1([0,1])$. Here $B_{\delta^{-n}}$ is a Euclidean ball of radius $\delta^{-n}$ and arbitrary centre, and $w_{B_{\delta^{-n}}}$ is a rapidly decaying weight function concentrated on $B_{\delta^{-n}}$; we refer to \cite{GGPRY2021} for the precise definitions. The inequality in the  $n=2$ case goes back to work of Fefferman \cite{Fefferman1973}. The general case is implicit in works of Prestini \cite{Prestini1984a, Prestini1984b}, albeit the arguments of these papers are somewhat lacking in detail. More recently, the inequality was rediscovered in \cite{GGPRY2021}, which includes a complete proof and contextualises the result in relation to recent developments in harmonic analysis and analytic number theory. It is remarked that a reverse form of \eqref{eq: euclid sf} holds as a simple consequence of a classical and elementary square function estimate due to Carleson (see, for instance, \cite{RdF1983}).

Interest in bounds such as \eqref{eq: euclid sf} has been spurred by the breakthrough work of Bourgain--Demeter--Guth \cite{BDG2016} which settled the long-standing main conjecture in Vinogradov's mean value theorem, a central problem in the theory of Diophantine equations. The approach in \cite{BDG2016} relied on establishing certain \textit{decoupling estimates} for the extension operator associated to the moment curve. These estimates are of a superficially similar form to the inequality in \eqref{eq: euclid sf}, although \eqref{eq: euclid sf} is much more elementary than the key estimate from \cite{BDG2016} and is not sufficient to prove the main conjecture. Nevertheless, in \cite{GGPRY2021} the authors discuss a general philosophy relating square function bounds to Diophantine equations. 




\subsection{Remarks on the proof}\label{subsec: proof} Recall that the key ingredient in the proof of Theorem~\ref{thm: sf} is Proposition~\ref{prop: main}. The latter is a natural non-archimedean analogue of Proposition 1.3 from \cite{GGPRY2021}. It is remarked that the arguments used in \cite{GGPRY2021} rely heavily on the order structure of the real line and break down completely in the non-archimedean setting. Consequently, to establish  Proposition~\ref{prop: main} we use a markedly different approach which is more algebraic in nature and involves the geometric analysis of sublevel sets, corresponding to a structural analysis of polynomial congruences. 

To describe the rudiments of our approach, we first consider the following reformulation of Proposition~\ref{prop: main} in the case where $K = \Q_p$ is the field of $p$-adic numbers. 

\begin{corollary}\label{cor: main ref} Let $n$, $a \in \N$ and $p$ be a rational prime such that $p > n \geq 2$. Suppose $(x_1, \dots, x_n)$, $(y_1, \dots,  y_n) \in \Z^n$ satisfy the congruence equations 
\begin{equation}\label{system: powers cong}
        x_1^j + \cdots + x_n^j \equiv y_1^j + \cdots + y_n^j   \mod p^{n a}  \qquad \textrm{for $1 \leq j \leq n$.}
\end{equation}
Then there exists a permutation $\sigma$ on $\{1, \cdots, n\}$ such that $x_j \equiv y_{\sigma(j)} \mod p^a$ for all $1 \leq j \leq n$.
\end{corollary}

 Corollary~\ref{cor: main ref} is easily verified for $a =1$.\footnote{Moreover, when $a = 1$ one need only assume \eqref{system: powers cong} holds with $p^n$ replaced with $p$.} Indeed, this case holds as a consequence of the classical Girard--Newton formul\ae\, together with uniqueness of factorisation of polynomials over the field $\Z/p\Z$: see, for example, \cite{MT2004}. To prove the general case of Corollary~\ref{cor: main ref}, we will still make use of the Girard--Newton formul\ae. However, we must now consider polynomials over the rings $\Z/p^a\Z$ for $a \geq 2$ and therefore cannot rely on uniqueness of factorisation.

The key tool used to overcome these issues is a refined version of the Phong--Stein--Sturm sublevel set decomposition \cite{PSS1999}, formulated over non-archimedean local fields. It can be viewed as a refined structural description of polynomial congruences, extending work of Chalk \cite{Chalk1990} which is valid for large values of $a$ and work of Stewart \cite{Stewart1991} for polynomials with a nonzero discriminant. This decomposition has been applied previously by the second author to study complete exponential sums and congruence equations \cite{Wright2011, Wright2020} and is recalled in Lemma~\ref{lem: PSS} below. A slightly curious feature of the argument is that we apply the sublevel set decomposition over a high degree field extension of $K$ rather than $K$ itself. 




\subsection{Archimedean fields} Our arguments can be adapted to work in the archimedean setting. As a consequence, we obtain a new proof of the Euclidean estimate \eqref{eq: euclid sf} for the moment curve. Moreover, we are also able to prove an analogue of \eqref{eq: euclid sf} when $K = {\mathbb C}$ is the complex field. In this case, $E$ is the extension operator associated to a certain $2$-surface in ${\mathbb R}^{2n} \simeq {\mathbb C}^n$. For $n=2$ this complex estimate is contained in \cite{BBH}, but it appears to be new in higher dimensions. Adapting the proofs to the archimedean setting is not entirely straightforward, and a discussion of the necessary modifications is provided in Appendix~\ref{appendix}. 




\subsection*{Notation} Depending on the context, $|A|$ will either denote the absolute value of a complex number $A$ or the cardinality of a finite set $A$.




\section{Review of the basic concepts from the theory of local fields}\label{sec: background}




\subsection{Non-archimedean local fields} A \textit{valued field } $(K, |\,\cdot\,|_K)$ is a field $K$ together with an \textit{absolute value} map $|\,\cdot\,|_K \colon K \to [0,\infty)$ satisfying
\begin{enumerate}
    \item[i)] $|x|_K = 0$ if and only if $x = 0$;
    \item[ii)] $|xy|_K = |x|_K|y|_K$ for all $x$, $y \in K$;
    \item[iii)] $|x + y|_K \leq |x|_K + |y|_K$  for all $x$, $y \in K$.
\end{enumerate} 
The absolute value $|\,\cdot\,|_K$ is \textit{non-archimedean} if iii) can be strengthened to 
\begin{enumerate}
    \item[iii$^\prime$)] $|x + y|_K \leq \max\{|x|_K, |y|_K\}$ for all $x, y \in K$,
\end{enumerate}
otherwise it is \textit{archimedean}. Note that any field $K$ admits a \textit{trivial absolute value} is given by $|x|_K = 1$ for all $x \in K^*$ (the group of units) and $|0|_K = 0$. 

A valued field $(K, |\,\cdot\,|_K)$ is endowed with a metric $d_K$ by setting $d_K(x,y) := |x-y|_K$ for all $x$, $y \in K$. For a non-archimedean absolute value $d$ is an ultrametric, satisfying the ultrametric triangle inequality $d_K(x,z) \leq \max \{d_K(x,y), d_K(y,z)\}$ for all $x$, $y$, $z \in K$. The \textit{ball centred at $x \in K$ of radius $r > 0$} is defined by 
\begin{equation*}
    B_K(x,r) := \{y \in K : |y - x|_K \leq r\}.
\end{equation*}

Henceforth let $(K, |\,\cdot\,|_K)$ be a valued field where $|\,\cdot\,|_K$ is a non-trivial, non-archimedean absolute. The \textit{ring of integers of $K$} is defined as 
\begin{equation*}
    \fo_K := \{x \in K : |x|_K \leq 1\};
\end{equation*} it is easy to see $\fo_K$ is a local ring with unique maximal ideal 
\begin{equation*}
    \fm_K := \{x \in K : |x|_K < 1\}.
\end{equation*} 
The \textit{residue class field} of $K$ is defined to be the quotient $k_K := \fo_K / \fm_K$. Finally, the \textit{value group} $\Gamma_K := \{|x|_K \in (0, \infty) : x \in K^*\}$ is the multiplicative subgroup of $(0,\infty)$ formed by the image of $K^*$ under $|\,\cdot\,|_K$. 

The absolute value $|\,\cdot\,|_K$ is \textit{discrete} if the group $\Gamma_K$ is discrete. This holds if and only if the maximal ideal $\fm_K$ is principal. In this case, we let $\pi_K \in \fm_K$ denote some choice of generator, which is referred to as a \textit{uniformiser} for $K$. It follows that $\Gamma_K = \{ q_K^{-\nu} : \nu \in \Z\}$ where $q_K := |\pi_K|_K^{-1} \in (1, \infty)$. 

\begin{definition} A valued field $(K, |\,\cdot\,|_K)$  is a \textit{non-archimedean local field} if $|\,\cdot\,|_K$ is a non-trivial discrete non-archimedean absolute value, it is complete and the residue class field $k_K$ is finite.
\end{definition}

If $(K, |\,\cdot\,|_K)$ is a non-archimedean local field, then $\fo_K$ is a compact subset of $K$ and, consequently, $K$ is a locally compact metric space. If we fix $\pi_K$ a uniformiser for $K$ and $\cA \subseteq \fo_K$ a set of representatives for $k_K$, then every $x \in K^*$ can be written uniquely as $x = \sum_{m = M}^{\infty} x_m \pi_K^m$ for some sequence $(x_m)_{m = M}^{\infty}$ of elements from $\mathcal{A}$. Here the series is understood to converge with respect to the metric $d$ introduced above. It follows that each ball $B_K(x, q_K^{-\nu})$, where $x \in K$ and $\nu \in \Z$, is the union of precisely $|k_K|$ balls of radius $q_K^{-\nu - 1}$. For further details, see \cite[Chapter 4]{Cassels}.




\subsection{Field extensions} 

Suppose $(K, |\,\cdot\,|_K)$ is a non-archimedean local field and $L : K$ is a finite extension of $K$ of degree $d \in \N$. Then there exists a unique extension $|\,\cdot\,|_L$ of $|\,\cdot\,|_K$ to $L$. Furthermore, $(L, |\,\cdot\,|_L)$ is also a non-archimedean local field. We say the extension $L : K$ is \textit{totally ramified} if the residue class fields $k_K$ and $k_L$ are isomorphic. In this case, if $\pi_K$ and $\pi_L$ are uniformisers of $K$ and $L$, respectively, then $|\pi_K|_K = |\pi_L|_L^d$. Thus, $\Gamma_L = \{q_K^{-\nu/d} : \nu \in \Z\}$ where $q_K := |\pi_K|_K^{-1}$. For further details, see \cite[Chapter 7]{Cassels}.

To construct a totally ramified extension of $(K, |\,\cdot\,|_K)$ of an arbitrary degree $d \in \N$, consider the polynomial $f \in K[X]$ given by $f(X) := X^d - \pi_K$. By Eisenstein's criterion (see \cite[Theorem 2.1]{Cassels}), $f$ is irreducible over $K$. Thus, if $\zeta$ a root of $f$, lying in the algebraic closure of $K$, then the simple extension $K(\zeta)$ has degree $d$ and is totally ramified by \cite[Theorem 7.1]{Cassels}.




\subsection{Vector spaces} Given a valued field $(K, |\,\cdot\,|_K)$ and $n \in \N$, the $n$-dimensional vector space $K^n$ over $K$ is endowed with the norm
\begin{equation*}
    |\bx|_K := \max\{|x_1|_K, \dots, |x_n|_K\} \qquad \textrm{for all $\bx = (x_1, \dots, x_n) \in K^n$.}
\end{equation*}
The \textit{ball centred at $\bx \in K^n$ of radius $r > 0$} is then defined by 
\begin{equation*}
    B(\bx,r) := \{\by \in K^n : |\by - \bx|_K \leq r\}.
\end{equation*}




\subsection{Fourier analysis on non-archimedean local fields} By the above discussion, any non-archimedean local field $(K, |\,\cdot\,|_K)$ is a LCA group and therefore admits an additive Haar measure $\mu$. By appropriately normalising, one may assume $\mu(\fo_K) = 1$.

Let $\widehat{K}$ denote the Pontryagin dual of $K$. There exists a character $e \in \widehat{K}$ with the property that the restriction of $e$ to $\fo_K$ is a principal character on the additive subgroup $\fo_K$ whilst the restriction of $e$ to $\pi_K^{-1}\fo_K$ is non-principal on the additive subgroup $\pi_K^{-1}\fo_K$. We will apply Fourier analysis over the vector spaces $K^n$, which are endowed with the Haar measure given by the $n$-fold product of $\mu$, also denoted by $\mu$. Given any $\bxi \in K^n$, if one defines $e_{\bxi} \colon K^n \to \C$ by $e_{\bxi}(\bx) := e(\bx \cdot \bxi)$ for $\bx \in K^n$ where $\bx \cdot \bxi := x_1\xi_1 + \cdots + x_n \xi_n$, then $\bxi \mapsto e_{\bxi}$ is an isomorphism between $K^n$ and $\widehat{K}^n$. For further details see \cite[Chapter 1, \S8]{Taibleson}.

Let $\nu$ be a Borel measure on $K^n$. By duality, we may also consider this as a measure on $\widehat{K}^n$ (in particular, this applies to the Haar measure). If $\nu$ is a finite measure, we may define the Fourier transform and inverse Fourier transform of $\nu$ by
\begin{equation*}
    \widehat{\nu}(\bxi) := \int_{K^n} \overline{e(\bx \cdot \bxi)}\,\ud \nu(\bx) \quad \textrm{and} \quad \widecheck{\nu}(\bx) := \int_{\widehat{K}^n} e(\bx \cdot \bxi)\,\ud \nu(\bxi)
\end{equation*}
With this definition, the rudiments of Fourier analysis such as the inversion formula, Parseval's theorem and Plancherel's theorem hold over $K^n$. For further details see \cite[Chapters 2-3]{Taibleson}.




\section{The proof of Proposition~\ref{prop: main}}\label{sec: proof}




\subsection{A structural lemma for sublevel sets} Central to the proof of Proposition~\ref{prop: main} is a non-archimedean structural decomposition for sublevel sets of univariate real polynomials due to Phong--Stein--Sturm~\cite{PSS1999}. Here we work in the abstract setting of a non-archimedean local field $(K, | \,\cdot\, |_K)$. The Phong--Stein--Sturm decomposition from \cite{PSS1999} was extended to such fields in \cite{Wright2011}, and we state this version below in Lemma~\ref{lem: PSS}. 

To introduce the key lemma, suppose $\bxi = (\xi_1, . . . , \xi_n) \in (\fo_K)^n$ is an $n$-tuple of roots in the ring of integers $\fo_K$ of $K$ and define the monic polynomial $P_{\bxi} \in K[X]$ by
\begin{equation}\label{eq: poly}
    P_{\bxi}(X) = \prod_{j=1}^n (X - \xi_j).
\end{equation}
Given $0 < \varepsilon \leq 1$, we are interested in analysing the structure of the sublevel sets
\begin{equation*}
    \{ z \in \fo_K : |P_{\bxi}(z)|_K \leq \varepsilon\}.
\end{equation*}
Naturally, this depends on the distribution of the roots $\xi_j$ and, to understand this, we consider \textit{root clusters} $\mathcal{C}$, which are simply defined to be subsets of $\{\xi_1, \dots, \xi_n\}$. 

\begin{lemma}[\cite{Wright2011}]\label{lem: PSS}
Suppose $(K,|\,\cdot\,|_K)$ is a non-archimedean local field and $\bxi = (\xi_1, . . . , \xi_n) \in (\fo_K)^n$ is an $n$-tuple of roots. For all $0 < \varepsilon \leq 1$ we have
\begin{equation*}
    \{ z \in \fo_K : |P_{\bxi}(z)|_K \leq \varepsilon\} =  \bigcup_{j=1}^n B_K(\xi_j;r_j(\bxi, \varepsilon))
\end{equation*}
where
\begin{equation}\label{eq: PSS radii}
    r_j(\bxi;\varepsilon) := \min_{\cC \ni \xi_j} \Big( \frac{\varepsilon}{\prod_{\xi_i \notin \cC} |\xi_j - \xi_i|_K} \Big)^{1/|\cC|}.
\end{equation}
Here the minimum is taken over all root clusters $\cC$ containing $\xi_j$.
\end{lemma}

\begin{remark}\label{rmk: coarse scale} By taking $\cC = \{\xi_1, \dots, \xi_n\}$ in the expression defining the radii in \eqref{eq: PSS radii}, we see that $r_j(\bxi;\varepsilon) \leq \varepsilon^{1/n}$ for $1 \leq j \leq n$.
\end{remark}

We will work with the following `self-referential' formula for the radii \eqref{eq: PSS radii}.

\begin{lemma}\label{self-ref lem} Let $\bxi$ and $r_j(\bxi; \varepsilon)$ be as in the statement of Lemma~\ref{lem: PSS}. For $1 \leq j \leq n$ define the root cluster
\begin{equation*}
    \mathcal{C}_j := B_K(\xi_j; r_j(\bxi, \varepsilon)) \cap \{\xi_1, \dots, \xi_n\}.
\end{equation*}
Then
\begin{equation*}
    r_j(\bxi;\varepsilon) = \Big( \frac{\varepsilon}{\prod_{\xi_i \notin \cC_j} |\xi_j - \xi_i|_K} \Big)^{1/|\cC_j|}.
\end{equation*}
\end{lemma}

\begin{proof} Fix $1 \leq j \leq n$ and let $\cC$ be a root cluster which achieves the minimum in \eqref{eq: PSS radii}, so that
\begin{equation}\label{eq: self-ref 1}
    r_j := r_j(\bxi;\varepsilon) = \Big( \frac{\varepsilon}{\prod_{\xi_i \notin \cC} |\xi_j - \xi_i|_K} \Big)^{1/|\cC|}. 
\end{equation}

Writing
\begin{equation*}
    \prod_{\xi_i \notin \cC} |\xi_j - \xi_i|_K = \prod_{\xi_i \in \cC_j \setminus \cC} |\xi_j - \xi_i|_K \prod_{\xi_i \notin \cC_j} |\xi_j - \xi_i|_K \prod_{\xi_i \in \cC \setminus \cC_j} |\xi_j - \xi_i|_K^{-1}
\end{equation*}
and using the fact that $|\xi_j - \xi_i|_K \leq r_j$ if and only if $\xi_i \in \cC_j$, we deduce that
\begin{equation}\label{eq: self-ref 2}
    \prod_{\xi_i \notin \cC} |\xi_j - \xi_i|_K \leq r_j^{|\cC_j \setminus \cC| - |\cC \setminus \cC_j|} \prod_{\xi_i \notin \cC_j} |\xi_j - \xi_i|_K.
\end{equation}
Combining \eqref{eq: self-ref 1} and \eqref{eq: self-ref 2} together with the elementary count 
\begin{equation*}
    |\cC| + |\cC_j \setminus \cC| - |\cC \setminus \cC_j| = |\cC_j|,
\end{equation*} 
we obtain
\begin{equation*}
  r_j = r_j^{(|\cC_j \setminus \cC| - |\cC \setminus \cC_j|)/|\cC_j|}  \Big( \frac{\varepsilon}{\prod_{\xi_i \notin \cC} |\xi_j - \xi_i|_K} \Big)^{1/|\cC_j|} \geq  \Big( \frac{\varepsilon}{\prod_{\xi_i \notin \cC_j} |\xi_j - \xi_i|_K} \Big)^{1/|\cC_j|}.
\end{equation*}
The desired identity immediately follows. 
\end{proof}

We emphasise that Lemmas \ref{lem: PSS} and \ref{self-ref lem} are valid in any non-archimedean local field. We will apply them to certain field extensions of the field $K$ appearing in Proposition \ref{prop: main}.




\subsection{The main argument} Here we apply the tools introduced in the previous subsection to prove Proposition~\ref{prop: main}.

\begin{proof}[Proof (of Proposition~\ref{prop: main})] The argument is broken into steps.\smallskip

\noindent \underline{Step 1.} Suppose $\bx = (x_1, \dots, x_n)$, $\by  = (y_1, \dots, y_n) \in (\fo_K)^n$ satisfy \eqref{system: powers}, so that
\begin{equation*}
        |p_j(\bx) - p_j(\by)|_K \leq  N^{-n}  \qquad \textrm{for $1 \leq j \leq n$,}
\end{equation*}
where the $p_j \in K[X_1, \dots, X_n]$ are the degree $j$ power sums $p_j(X) = \sum_{\ell=1}^n X_\ell^j$ for $1 \leq j \leq n$  and $N := (q_K)^a$ for some $a \in \N$. Then $\bx$, $\by$ also satisfy
\begin{equation}\label{eq: main prop 1 1}
        |e_j(\bx) - e_j(\by)|_K \leq  N^{-n}  \qquad \textrm{for $1 \leq j \leq n$,}
\end{equation}
where the $e_j \in K[X_1, \dots, X_n]$ are the the degree $j$ elementary symmetric polynomials $e_j(X) = \sum_{k_1 < \cdots <k_j}X_{k_1}\cdots X_{k_j}$ for $1 \leq j \leq n$. Indeed, this is a direct consequence of the Girard--Newton formul\ae
\begin{equation*}
    je_j(X_1,\ldots,X_n) = \sum_{i=1}^j(-1)^{i-1} e_{j - i} (X_1, \ldots, X_n) p_i(X_1, \ldots, X_n) \qquad \textrm{$1 \leq j \leq n$}, 
\end{equation*}
since the hypothesis $\mathrm{char}\,k_K > n$ ensures $|j|_K = 1$ for $1 \leq j \leq n$.
\smallskip

\noindent \underline{Step 2.} Given an $n$-tuple of roots $\bxi = (\xi_1, \dots, \xi_n) \in (\fo_K)^n$, define the polynomial $P_{\bxi}  \in K[X]$ as in \eqref{eq: poly}. In particular,
\begin{equation}\label{eq: main prop 2 1}
    P_{\bxi}(X) = \prod_{j=1}^n (X - \xi_j) = \sum_{j=0}^n (-1)^{n-j}e_{n-j}(\bxi)X^j
\end{equation}

Let $K_{\circ} : K$ be a finite extension, and $|\,\cdot\,|_{K_{\circ}}$ the unique extension of $|\,\cdot\,|_K$ to $K_{\circ}$. We can then interpret $P_{\bxi}$ as lying in the polynomial ring $K_{\circ}[X]$. Moreover, for $\bx$, $\by$ as in Step 1, it then follows that
\begin{equation}\label{eq: main prop 2 2}
    \big\{ z \in \fo_{K_{\circ}} : |P_{\bx}(z)|_{K_{\circ}} \leq N^{-n} \big\} = \big\{ z \in \fo_{K_{\circ}} : |P_{\by}(z)|_{K_{\circ}} \leq N^{-n} \big\}.
\end{equation}
To see this, we note \eqref{eq: main prop 2 1}, \eqref{eq: main prop 1 1} and the ultrametric triangle inequality imply
\begin{equation*}
    |P_{\bx}(z) - P_{\by}(z)|_{K_{\circ}} \leq \max_{0 \leq j\leq n} |e_{n-j}(\bx) - e_{n-j}(\by)|_K |z|_{K_{\circ}}^j \leq N^{-n} \qquad \textrm{for all $z \in \fo_{K_{\circ}}$.}
\end{equation*}
The desired identity \eqref{eq: main prop 2 2} now follows from another application of the ultrametric triangle inequality.\medskip

\noindent \underline{Step 3.} In remaining steps we will analyse the structure of the sublevel sets featured in \eqref{eq: main prop 2 2} in order to determine information about $\bx$, $\by$. We will carry out this analysis at two separate scales: a course scale, introduced here in Step 3, and a finer scale which is analysed in the remaining steps. 

By the Phong--Stein--Sturm sublevel set decomposition from Lemma~\ref{lem: PSS}, and in particular the observation in Remark~\ref{rmk: coarse scale}, for any $\bxi = (\xi_1, \dots, \xi_n) \in (\fo_{K_{\circ}})^n$ we have
\begin{equation}\label{eq: main prop 3 1}
    \big\{ z \in \fo_{K_{\circ}} : |P_{\bxi}(z)|_{K_{\circ}} \leq N^{-n} \big\} \subseteq \bigcup_{j=1}^n B_{K_{\circ}}(\xi_j, N^{-1}). 
\end{equation}
From this and \eqref{eq: main prop 2 2}, we see that:
\begin{itemize}
    \item For all $1 \leq j \leq n$ there exists some $1 \leq j' \leq n$  such that $|x_j - y_{j'}|_K \leq N^{-1}$;
    \item For all $1 \leq j \leq n$ there exists some $1 \leq j' \leq n$ such that $|y_j - x_{j'}|_K \leq N^{-1}$.
\end{itemize}
This sets up a bipartite graph $G = (X, Y, E)$ where the vertex sets $X := \{x_1, \dots, x_n\}$ and $Y := \{y_1, \dots, y_n\}$ are formed by the components of $\bx$ and $\by$ and $x_i \in X$ and $y_j \in Y$ are adjacent if and only if $|x_i - y_j|_K \leq N^{-1}$. It follows from the above that there are no isolated vertices. Furthermore, the ultrametric property implies the connected components $G_1, \dots, G_M$ of $G$ are complete bipartite graphs.

Write $G_m = (U_m, V_m, E_m)$ for $1 \leq m \leq M$. The vertex sets $U_m \subseteq X$ and $V_m \subseteq Y$ are referred to as \textit{superclusters}. We let $\alpha_m := |U_m|$ and $\beta_m := |V_m|$. In light of the above, the problem is reduced to showing 
\begin{equation}\label{eq: main prop 3 2}
    \alpha_m = \beta_m \qquad \textrm{for $1 \leq m \leq M$.}
\end{equation} 
Indeed, if this is the case, then we can define a permutation $\sigma$ on $\{1, \dots, n\}$ such that if $x_j \in U_m$ for some $1 \leq j \leq n$ and $1 \leq m \leq M$, then $y_{\sigma(j)} \in V_m$. By the properties of the superclusters, it follows that $|x_j - y_{\sigma(j)}|_K \leq N^{-1}$ for all $1 \leq j \leq n$. \medskip

\noindent \underline{Step 4.} To prove \eqref{eq: main prop 3 2}, we argue by contradiction. Suppose there exists some $1  \leq m \leq M$ such that $\alpha_m \neq \beta_m$. By relabelling, we may assume without loss of generality that $\beta_1 > \alpha_1$ and, moreover, that $\beta_1/\alpha_1 > 1$ maximises the ratio $\beta_m / \alpha_m$ over all choices of $1 \leq m \leq M$. 

We now analyse the problem at a smaller scale, within the superclusters $U_m$ and $V_m$. Refining \eqref{eq: main prop 3 1}, we know from \eqref{eq: main prop 2 2} and Lemma~\ref{lem: PSS} that
\begin{equation*}
    \bigcup_{j=1}^n B_{K_{\circ}}(x_j, r_j(\bx, N^{-n})) =  \bigcup_{j=1}^n B_{K_{\circ}}(y_j, r_j(\by, N^{-n})).
\end{equation*}

Our first observation is that if $x_j \in U_m$ and $y_{j'} \in V_{m'}$ for $m \neq m'$, then the balls $B_{K_{\circ}}(x_j; r_j(\bx, N^{-n}))$ and $B_{K_{\circ}}(y_{j'}; r_{j'}(\by, N^{-n}))$ are disjoint. This allows us to home in and analyse the superclusters $U_1$, $V_1$ individually. 

To simplify notation, for $u \in U_1$ and $v \in V_1$, write $r_X(u) := r_i(\bx;N^{-n})$ and $r_Y(v) := r_j(\by;N^{-n})$ where $1 \leq i,j \leq n$ is such that $u = x_i$ and $v = y_j$. By the above observations,
\begin{equation*}
    \bigcup_{u \in U_1} B_{K_{\circ}}(u; r_X(u)) = \bigcup_{v \in V_1} B_{K_{\circ}}(v; r_Y(v)).
\end{equation*}
We now apply an ultrametric version of the Vitali cover procedure to pass to disjoint families of balls. In particular, there exist subcollections $U \subseteq U_1$ and $V \subseteq V_1$ such that the collections of balls
\begin{equation*}
    \big\{ B_{K_{\circ}}(u, r_X(u)) : u \in U \big\} \quad \textrm{and} \quad \big\{ B_{K_{\circ}}(v, r_Y(v)) : v \in V \big\}
\end{equation*}
are pairwise disjoint and
\begin{equation*}
 \bigcup_{u \in U} B_{K_{\circ}}(u; r_X(u)) =   \bigcup_{u \in U_1} B_{K_{\circ}}(u; r_X(u)) = \bigcup_{v \in V_1} B_{K_{\circ}}(v; r_Y(v)) = \bigcup_{v \in V} B_{K_{\circ}}(v; r_Y(v)).
\end{equation*}

At this point, we assume our ambient field $K_{\circ}$ is a totally ramified finite extension of $K$. Under this hypothesis, the residue class field of $k_{K_{\circ}}$ is isomorphic to $k_K$. In particular, since by hypothesis $|k_K| \geq \mathrm{char}\,k_K > n$, any ball $B_{K_{\circ}}(x,r)$ cannot be written as a union of $n$ (not necessarily distinct) balls with strictly smaller radii. Consequently, $|U| = |V|$ and there exists enumerations of the sets $U = \{u_1, \dots, u_L\}$, $V = \{v_1, \dots, v_L\}$ such that
\begin{equation}\label{eq: main prop 4 1}
    B_{K_{\circ}}(u_{\ell}, r_X(u_{\ell})) = B_{K_{\circ}}(v_{\ell}, r_Y(v_{\ell})) \qquad \textrm{for $1 \leq \ell \leq L$.}
\end{equation}

At this stage, we wish to conclude that
\begin{equation}\label{eq: main prop 4 2}
    r_X(u_{\ell}) = r_Y(v_{\ell}) \qquad \textrm{for $1 \leq \ell \leq L$.}
\end{equation}
If we work with $K_{\circ} = K$ in \eqref{eq: main prop 2 2}, then \eqref{eq: main prop 4 2} does not necessarily follow from \eqref{eq: main prop 4 1} owing to the discrete nature of the value group. To address this, we now further assume that $K_{\circ} \colon K$ is a degree $n!$ totally ramified extension. Under this hypothesis, the value groups $\Gamma_K$ and $\Gamma_L$ take the form
\begin{equation*}
\Gamma_K = \{q_K^{-\nu} : \nu \in \Z\} \qquad \textrm{and} \qquad \Gamma_{K_{\circ}} = \{q_K^{-\nu/n!} : \nu \in \Z\}. 
\end{equation*}
In particular, $\Gamma_{K_{\circ}}$ contains the quantities $r_X(u_{\ell})$ and $r_Y(v_{\ell})$. Thus, working over $\fo_{K_{\circ}}$, we may deduce \eqref{eq: main prop 4 2} from \eqref{eq: main prop 4 1}.\medskip

\noindent \underline{Step 5.} We now apply the self-referential form of the Phong--Stein--Sturm sublevel decomposition from Lemma~\ref{self-ref lem} to obtain a formula for the radii appearing in \eqref{eq: main prop 4 2}. For $1 \leq \ell \leq L$, let 
\begin{equation*}
    \cC_X(u_{\ell}) := B_{K_{\circ}}(u_{\ell}, r_X(u_{\ell})) \cap X \quad \textrm{and} \quad \cC_Y(v_{\ell}) := B_{K_{\circ}}(v_{\ell}, r_Y(v_{\ell})) \cap Y
\end{equation*}
denote the clusters appearing in Lemma~\ref{self-ref lem}, which realise the minimum in \eqref{eq: PSS radii}. 

We first consider the contributions to the radii arising from roots in superclusters other than $U_1$ and $V_1$. By the ultrametric property, for each $2 \leq m \leq M$ there exists some $D_m > N^{-1}$ such that 
\begin{equation*}
    |u_{\ell} - u'|_K = |v_{\ell} - v'|_K = D_m \qquad \textrm{for $1 \leq \ell \leq L$ and $u' \in U_m$, $v' \in V_m$.}
\end{equation*}
Consequently, recalling the definition of the $\alpha_m$ and $\beta_m$ from Step 3, we have
\begin{equation}\label{eq: main prop 5 1}
  \prod_{u' \in X \setminus U_1} |u_{\ell} - u'|_K =  \prod_{m=2}^M D_m^{\alpha_m} \quad \textrm{and} \quad \prod_{v' \in Y \setminus V_1} |v_{\ell} - v'|_K =  \prod_{m=2}^M D_m^{\beta_m}.
\end{equation}

We now turn to the contributions of roots within $U_1$ and $V_1$. For $1 \leq \ell, \ell' \leq L$ with $\ell \neq \ell'$ we have
\begin{equation*}
    |u_{\ell} - u'|_K = |u_{\ell} - u_{\ell'}|_K = |v_{\ell} - v_{\ell'}|_K = |v_{\ell} - v'|_K  \quad \textrm{for all $u' \in \cC_X(u_{\ell'})$, $v' \in \cC_Y(v_{\ell'})$.}
\end{equation*}
 In particular, if we define $s_{\ell} := |\cC_X(u_{\ell})|$ and $t_{\ell} := |\cC_Y(v_{\ell})|$ for $1 \leq \ell \leq L$, it follows that 
\begin{equation}\label{eq: main prop 5 2}
  \prod_{\substack{u' \in U_1 \\ u' \notin \cC_X(u_{\ell})}} |u_{\ell} - u'|_K =  \prod_{\substack{1 \leq \ell' \leq L \\ \ell' \neq \ell}} |u_{\ell} - u_{\ell'}|_K^{s_{\ell'}} \quad \textrm{and } \prod_{\substack{v' \in V_1 \\ v' \notin \cC_Y(v_{\ell})}} |v_{\ell} - v'|_K =  \prod_{\substack{1 \leq \ell' \leq L \\ \ell' \neq \ell}} |v_{\ell} - v_{\ell'}|_K^{t_{\ell'}}
\end{equation}
whilst we also have
\begin{equation}\label{eq: main prop 5 3}
    s_1 + \cdots + s_L = \alpha_1 < \beta_1 = t_1 + \cdots + t_L.
\end{equation}

Combining Lemma~\ref{self-ref lem} with \eqref{eq: main prop 5 1} and \eqref{eq: main prop 5 2}, and applying the identity \eqref{eq: main prop 4 1}, we conclude that
\begin{equation}\label{eq: main prop 5 4}
    \bigg(\frac{N^{-n}}{\prod_{\substack{1 \leq \ell' \leq L \\ \ell' \neq \ell}} |u_{\ell} - u_{\ell'}|_K^{s_{\ell'}} \prod_{m=2}^M D_m^{\alpha_m}}\bigg)^{1/s_{\ell}} = \bigg(\frac{N^{-n}}{\prod_{\substack{1 \leq \ell' \leq L \\ \ell' \neq \ell}} |v_{\ell} - v_{\ell'}|_K^{t_{\ell'}} \prod_{m=2}^M D_m^{\beta_m}}\bigg)^{1/t_{\ell}}
\end{equation}
for all $1 \leq \ell \leq L$. Thus, raising the above display to the $s_{\ell} t_{\ell}$ power and rearranging the resulting expression gives
\begin{equation}\label{eq: main prop 5 5}
    N^{-n(t_{\ell} - s_{\ell})} \prod_{\substack{1 \leq \ell' \leq L \\ \ell' \neq \ell}} |u_{\ell} - u_{\ell'}|_K^{-(s_{\ell'}t_{\ell} - s_{\ell}t_{\ell'})} = \prod_{m=2}^M D_m^{t_{\ell}\alpha_m - s_{\ell'}\beta_m}.
\end{equation}
Taking the product of either side of the identity \eqref{eq: main prop 5 5} over all choices of $\ell$, we deduce from \eqref{eq: main prop 5 3} that 
\begin{equation*}
 N^{-n(\beta_1 - \alpha_1)} \prod_{\substack{1 \leq \ell, \ell' \leq L \\ \ell' \neq \ell}} |u_{\ell} - u_{\ell'}|_K^{-(s_{\ell'}t_{\ell} - s_{\ell}t_{\ell'})} = \prod_{m=2}^M D_m^{\beta_1\alpha_m - \alpha_1\beta_m}
\end{equation*}
and therefore, by parity considerations,
\begin{equation*}
     N^{-n(\beta_1 - \alpha_1)} = \prod_{m=1}^M D_j^{\beta_1\alpha_m - \alpha_1\beta_m}.
\end{equation*}
From our labelling of the superclusters, we know $\beta_1/\alpha_1 \geq \beta_m/\alpha_m$ for all $1 \leq m \leq M$. Furthermore, since $\alpha_1 + \cdots + \alpha_M = \beta_1 + \cdots + \beta_M = n$ and $\beta_1/\alpha_1 > 1$, there must exist at least one choice of $m$ for which $\beta_1/\alpha_1 > \beta_m/\alpha_m$ (that is, the inequality is strict). Consequently, all of the exponents $\beta_1\alpha_m - \alpha_1\beta_m$ are non-negative and at least one exponent is strictly positive. Thus, since $D_m > N^{-1}$ for $1 \leq m \leq M$, we conclude that
\begin{equation*}
    N^{-n(\beta_1 - \alpha_1)} > N^{-n(\beta_1 - \alpha_1)},
\end{equation*}
which is a contradiction. This arises from the assumption that \eqref{eq: main prop 3 2} fails, and so \eqref{eq: main prop 3 2} must hold, concluding the proof.
\end{proof}




\section{The C\'ordoba--Fefferman argument}\label{sec: Cordoba Fefferman} In this section we apply the standard C\'ordoba--Fefferman argument \cite{Fefferman1973} in order to obtain Theorem~\ref{thm: sf} from Proposition~\ref{prop: main}. 

\begin{proof}[Proof (of Theorem~\ref{thm: sf})] By translation invariance, we may assume $\bx = 0$, Letting $\delta := q_K^{-\alpha}$ and $\varphi := \chi_{B_{\delta^{-n}}}$ denote the characteristic function of the ball $B_{\delta^{-n}} := B(0,q_K^{\alpha n})$, we have
\begin{equation*}
    |Ef|^{2m}\cdot \varphi = \sum_{\substack{I_j, J_j \in \mathcal{I}(\delta) \\ 1 \leq j \leq m}} \prod_{j=1}^m Ef_{I_j} \cdot \varphi \;\overline{\prod_{j=1}^m Ef_{J_j}\cdot \varphi}.
\end{equation*}
Thus, by Parseval's theorem,
\begin{equation}\label{eq: CF 1}
    \|Ef\|_{L^{2m}(B_{\delta^{-n}})}^{2m} = \sum_{\substack{I_j, J_j \in \mathcal{I}(\delta) \\ 1 \leq j \leq m}} \int_{K^n} \Big(\prod_{j=1}^m Ef_{I_j} \cdot \varphi\Big)\;\widehat{}\;(\bxi)\; \overline{\Big(\prod_{j=1}^m Ef_{J_j}\cdot \varphi\Big)\;\widehat{}\;(\bxi)}\,\ud \mu(\bxi).
\end{equation}

Let $\nu$ denote the pushforward of the Haar measure on $\fo_K$ under the moment mapping $\gamma \colon \fo_K \to \fo_K^n$ given by $\gamma(t) := (t, t^2, \dots, t^n)$ for all $t \in \fo_K$. Observe that 
\begin{equation*}
    Eg = (\,g \ud \nu\,)\;\widecheck{}\; \qquad \textrm{for $g \in L^1(\fo_K)$}
\end{equation*}
and so $(E f_I \cdot \varphi)\;\widehat{}\; = \widehat{\varphi}\; \ast f_I\ud \nu$ for any $I \in \mathcal{I}(\delta)$. Thus, fixing $I_j$, $J_j \in \mathcal{I}(\delta)$ for $1 \leq j \leq n$, it follows that the right-hand integrand in \eqref{eq: CF 1} can be written as 
\begin{equation}\label{eq: CF 2}
     \big(\widehat{\varphi}\; \ast f_{I_1}\ud \nu\big) \ast \cdots \ast \big(\widehat{\varphi}\; \ast f_{I_m}\ud \nu\big) (\bxi) \; \overline{\big(\widehat{\varphi}\; \ast f_{J_1}\ud \nu\big) \ast \cdots \ast \big(\widehat{\varphi}\; \ast f_{J_n}\ud \nu\big)(\bxi)}.
\end{equation}

By a simple computation, $\widehat{\varphi} = \delta^{-n^2} \chi_{B(0,\delta^n)}$ and, in particular, 
\begin{equation*}
   \mathrm{supp}\,(\widehat{\varphi} \ast f_I\ud \nu) \subseteq \big\{ \bxi \in \widehat{K}^n : |\bxi - \gamma(s)|_K \leq \delta^n \textrm{ for some $s \in I$} \big\} \qquad \textrm{for $I \in \mathcal{I}(\delta)$}.
\end{equation*}
Moreover, if $\bxi \in \widehat{K}^n$ lies in the support of the function in \eqref{eq: CF 2}, then
\begin{equation*}
    \Big|\bxi - \sum_{j=1}^m\gamma(s_j)\Big|_K \leq \delta^n  \textrm{ and }  \Big|\bxi - \sum_{j=1}^m\gamma(t_j)\Big|_K \leq \delta^n \textrm{ for some $s_j \in I_j$, $t_j \in J_j$, $1 \leq j \leq m$.}
\end{equation*}

Now suppose the support of the function in \eqref{eq: CF 2} is non-empty for some choice of $I_j$, $J_j \in \mathcal{I}(\delta)$ for $1 \leq j \leq m$. By the preceding observations, there must exist $s_j \in I_j$, $t_j \in J_j$ for $1 \leq j \leq m$ such that
\begin{equation*}
    \Big|\sum_{j=1}^m\gamma(s_j) - \sum_{j=1}^m\gamma(t_j)\Big|_K \leq q_K^{-\alpha n}. 
\end{equation*}
Applying Proposition~\ref{prop: main}, there exists a permutation $\sigma$ on $\{1,\cdots, m\}$ such that $|t_j - s_{\sigma(j)}|_K \leq q_K^{-\alpha}$ for all $1 \leq j \leq m$. By the ultrametric property, this can only happen if $J_j = I_{\sigma(j)}$ for all $1 \leq j \leq m$.  

In light of the discussion of the previous paragraph, we see that all the `off-diagonal' terms of the right-hand sum in \eqref{eq: CF 1} are zero and, in particular, 
\begin{align*}
    \|Ef\|_{L^{2m}(B_{\delta^{-n}})}^{2m} &\leq m! \sum_{I_1, \dots, I_m \in \mathcal{I}(\delta)} \int_{K^n} \Big|\Big(\prod_{j=1}^m Ef_{I_j} \cdot \varphi\Big)\;\widehat{}\;(\bxi)\;\Big|^2\,\ud \mu(\bxi) \\
    &\leq m!  \Big\|\big( \sum_{I \in \mathcal{I}(\delta)} |Ef_I|^2\big)^{1/2}\Big\|_{L^{2k}(B_{\delta^{-n}})}^{2m} 
\end{align*}
where the second identity is a consequence of Plancherel's theorem. This concludes the proof. 
\end{proof}

\appendix




\section{Adapting the argument to archimedean local fields}\label{appendix}




\subsection{Key ingredients}

In this section we sketch how the arguments of \S\ref{sec: proof} can be adapted to work in the archimedean setting. The main result is as follows. 

\begin{proposition}\label{prop: arch main} Let $(K, |\,\cdot\,|_K)$ be an archimedean local field and $n \in \N$. There exists a constant $C_n \geq 1$, depending only on $n$, such that the following holds. Let $N \geq 1$ and suppose $(x_1, \dots, x_n)$, $(y_1, \dots,  y_n) \in B_K(0,1)^n$ satisfy
\begin{equation*}
   |x_1^j + \cdots + x_n^j - y_1^j - \cdots - y_n^j|_K \leq N^{-n} \qquad \textrm{for $1 \leq j \leq n$.}
\end{equation*}
Then there exists a permutation $\sigma$ on $\{1, \cdots, n\}$ such that $|x_j - y_{\sigma(j)}|_K \leq C_n N^{-1}$ for all $1 \leq j \leq n$.
\end{proposition}

Proposition~\ref{prop: arch main} can be combined with the C\'ordoba--Fefferman argument described in \S\ref{sec: Cordoba Fefferman} to yield the analogue of Theorem~\ref{thm: sf} for archimedean local fields. The proof of Proposition~\ref{prop: arch main} closely follows that of Proposition~\ref{prop: main}, albeit with a few minor points of divergence. Here we review the key tools used in the argument and how they differ from those used in the non-archimedean setting.\medskip

\noindent\textit{The Phong--Stein--Sturm decomposition.} An approximate version of Lemma~\ref{lem: PSS} holds over archimedean local fields. Moreover, if $(K, |\,\cdot\,|_K)$ is any valued field with non-trivial absolute value and $\bxi = (\xi_1, . . . , \xi_n) \in B_K(0,1)^n$ is an $n$-tuple of roots, then for all $0 < \varepsilon \leq 1$ we have
\begin{equation*}
  \bigcup_{j=1}^n B_K(\xi_j; 2^{-n} r_j(\bxi, \varepsilon))  \subseteq  \{ z \in K : |P_{\bxi}(z)| \leq \varepsilon\} \subseteq  \bigcup_{j=1}^n B_K(\xi_j; 2^n r_j(\bxi, \varepsilon));
\end{equation*}
see \cite[Proposition 3.3]{KW2012}. We also note that for any $\lambda \geq 1$ we have
\begin{equation*}
    \lambda^{1/n} r_j(\bxi, \varepsilon) \leq  r_j(\bxi, \lambda \varepsilon) \leq \lambda r_j(\bxi, \varepsilon). 
\end{equation*}

Using the above observations, Steps 1 - 3 in the proof of Proposition~\ref{prop: main} can be carried over in a straightforward manner to the archimedean setting, with additional constant factors appearing throughout the argument. In contrast with the non-archimedean case, we work directly over the field $K$ rather than some field extension (indeed, no rich theory of field extensions is available in the archimedean setting). The superclusters are defined using the condition $|x - y|_K \leq \rho N^{-1}$, where $\rho \geq 1$ is a parameter which is chosen large, depending only on $n$, so as to force a contradiction at the end of the argument.\medskip

\noindent\textit{A Vitali-type covering lemma.} Step 4 of the proof of Proposition~\ref{prop: main} featured an application of the ultrametric Vitali covering lemma, which was used to pass to the two identical families of balls in \eqref{eq: main prop 4 1}. The ultrametric covering lemma is very clean, owing to the fact that any two balls in an ultrametric space are either nested or disjoint. To adapt the argument to the archimedean setting, we make use of the following somewhat technical variant of the original Vitali covering lemma. 

\begin{lemma}\label{lem: Vitali variant} Let $\cB_X$, $\cB_Y$ be finite sets of closed balls in $\R^d$ of cardinality at most $n \in \N$. Suppose that $\lambda \geq 1$ is such that 
\begin{equation}\label{eq: Vitali lem}
    \bigcup_{B_X \in \cB_X} B_X \subseteq \bigcup_{B_Y \in \cB_Y} \lambda \cdot B_Y \quad \textrm{and} \quad \bigcup_{B_Y \in \cB_Y} B_Y \subseteq \bigcup_{B_X \in \cB_X} \lambda \cdot B_X.
\end{equation}
Then there exist $\cB_X' = \{B_X^1, \dots, B_X^L\} \subseteq \cB_X$ and $\cB_Y' = \{B_Y^1, \dots, B_Y^L\} \subseteq \cB_Y$ and a constant $R = R(n, \lambda)$ such that the following hold:
\begin{enumerate}[1)]
\item \textbf{Strong separation.} For all $1 \leq \ell < \ell' \leq L$, we have
\begin{equation*}
    2R\cdot B_X^{\ell} \cap 2R \cdot B_X^{\ell'} = \emptyset \quad \textrm{and} \quad 2R \cdot B_Y^{\ell} \cap 2R \cdot B_Y^{\ell'} = \emptyset.
\end{equation*}
    \item \textbf{Vitali covering.}  
    \begin{equation*}
      \displaystyle \bigcup_{B_X \in \cB_X} B_X \subseteq \bigcup_{\ell=1}^L R \cdot B_X^{\ell} \quad \textrm{and} \quad  \bigcup_{B_Y \in \cB_Y} B_Y \subseteq \bigcup_{\ell=1}^L R \cdot B_Y^{\ell} 
    \end{equation*}
    
    \item \textbf{Comparable balls.} For all $1 \leq \ell \leq L$, we have
    \begin{equation*}
        B_X^{\ell} \subseteq R \cdot B_Y^{\ell} \quad \textrm{and} \quad B_Y^{\ell} \subseteq R \cdot B_X^{\ell}.
    \end{equation*}
\end{enumerate}
\end{lemma}

Here, given a ball $B \subseteq \R^d$ and $\lambda > 0$, we let $\rad B$ denote the radius of $B$ and $\lambda \cdot B$ denote the ball concentric to $B$ but with radius $\lambda \rad B$. In applying the lemma, we identify the archimedean local field $\C$ with the metric space $\R^2$. 

\begin{remark} A key feature of Lemma~\ref{lem: Vitali variant} is that the parameter $R$ is allowed to depend on the number of balls $n$ (in stark contrast with the classical Vitali covering lemma). This flexibility allows for the comparability between the balls $B_X^{\ell}$ and $B_Y^{\ell}$. It also allows for the strong separation property in 1), where the separation parameter $2R$ is \textit{larger} than the dilation parameter in 2). 
\end{remark}

The comparability property 3) is a surrogate for the identification between balls in \eqref{eq: main prop 4 1} in the archimedean setting. Similarly, the strong separation property 1) is used to establish approximate versions of the identities in \eqref{eq: main prop 5 2}. 

Since Lemma~\ref{lem: Vitali variant} is a new feature of the argument, we present the full proof in \S\ref{sec: Vitali} below.\medskip  

\noindent\textit{The self-referential formula for the radii.} The final ingredient we highlight from the proof of Proposition~\ref{prop: main} is the self-referential formula for the radii from Lemma~\ref{self-ref lem}; recall, this is used to establish the identity \eqref{eq: main prop 5 4} in Step 5. The proof of Lemma~\ref{self-ref lem} does not rely on the ultrametric triangle inequality, and the result remains valid as stated in any valued field $(K, |\,\cdot\,|_K)$ with non-trivial absolute value. However, for the proof of Proposition~\ref{prop: arch main} we require a slight extension of the formula. For $1 \leq j \leq n$ and $\lambda \geq 1$ define the root cluster
\begin{equation*}
    \mathcal{C}_{j,\lambda} := B_K(\xi_j; \lambda r_j(\bxi, \varepsilon)) \cap \{\xi_1, \dots, \xi_n\}.
\end{equation*}
Then the proof of Lemma~\ref{self-ref lem} shows that 
\begin{equation*}
    r_j(\bxi;\varepsilon) \leq \Big( \frac{\varepsilon}{\prod_{\xi_i \notin \cC_{j,\lambda}} |\xi_j - \xi_i|_K} \Big)^{1/|\cC_{j,\lambda}|} \leq \lambda r_j(\bxi;\varepsilon). 
\end{equation*}
The approximate formula can be used to establish an approximate version of \eqref{eq: main prop 5 4}.




\subsection{Proof of the Vitali-type lemma}\label{sec: Vitali} In this section, we prove Lemma~\ref{lem: Vitali variant}. The first step is the following simple consequence of the classical Vitali covering lemma.

\begin{lemma}\label{lem: strong sep} Let $\cB$ be a finite collection of balls in $\R^d$ of cardinality at most $n$ and $R \geq 1$. Then there exists a subcollection $\cB \subseteq \cB$ and a constant $\lambda = \lambda(n, R) \geq 1$ depending only on $n$ and $R$ such that
\begin{enumerate}[1)]
    \item \textbf{Strong separation.} The large dilates $\{ \lambda R \cdot B' : B' \in \cB'\}$ are pairwise disjoint.\smallskip
    \item \textbf{Vitali covering.} The small dilates $\{ \lambda  \cdot B' : B' \in \cB'\}$ form a Vitali cover in the sense that
    \begin{equation*}
        \bigcup_{B \in \cB} B \subseteq \bigcup_{B' \in \cB'} \lambda \cdot B'.
    \end{equation*}
\end{enumerate}
\end{lemma}

\begin{proof}
\noindent The proof is based on repeated application of the classical Vitali covering lemma and pigeonholing. Starting with $\cB_0 := \cB$, we recursively construct a chain of proper subsets $\cB_m \subset \cB_{m-1} \subset \dots \subset \cB_0$ such that
\begin{equation}\label{eq: strong sep 1a}
    \bigcup_{B \in \cB} B \subseteq \bigcup_{B_m \in \cB_m} \lambda_m \cdot B  
\end{equation}
where $\lambda_m := (3R)^m$.

Suppose that $\cB_m$ has already been constructed and satisfies \eqref{eq: strong sep 1a}. Apply the classical Vitali covering lemma to the collection of dilated balls $\{R \lambda_m \cdot B : B \in \cB_m\}$ to obtain a subcollection $\cB_{m+1} \subseteq \cB_m$ such that 
\begin{equation*}
    \{R \lambda_m B : B \in \cB_{m+1}\} \quad \textrm{are pairwise disjoint}
\end{equation*}
and, noting $\lambda_{m+1} = 3R\lambda_m$, 
\begin{equation*}
    \bigcup_{B \in \cB} B \subseteq \bigcup_{B \in B_m} R\lambda_m \cdot B \subseteq \bigcup_{B \in \cB_{m+1}} \lambda_{m+1} \cdot B.
\end{equation*}
If $\cB_{m+1} = \cB_m$, then the algorithm terminates; otherwise, $\cB_{m+1} \subset \cB_m$ is a \textit{proper} subset, as required. 

By pigeonholing, the algorithm must terminate after at most $n-1$ steps. If $0 \leq M \leq n-1$ is the terminal step, then the desired properties hold with $\cB' := \cB_M$ and $\lambda := \lambda_M$. 
\end{proof}

We now turn to the proof of Lemma~\ref{lem: Vitali variant}. For a pair of balls $B_1$, $B_2 \subseteq \R^d$ we frequently make use of the following consequence of triangle inequality: 
\begin{equation}\label{eq: triangle ineq}
    \textrm{If $B_1 \cap B_2 \neq \emptyset$, then $B_1 \subseteq \lambda \cdot B_2$ where $\rad(\lambda\cdot B_2) = 2\rad B_1 + \rad B_2$.} 
\end{equation}
Note, in particular, that the dilate $\lambda \cdot B_2$ in the above display satisfies 
\begin{equation*}
    \rad(\lambda\cdot B_2) \leq 3 \max\{ \rad B_1, \rad B_2\}.
\end{equation*}

\begin{proof}[Proof (of Lemma~\ref{lem: Vitali variant})] We first note that it suffices to construct families $\cB_X'$ and $\cB_Y'$ satisfying properties 2) and 3) only. Indeed, once this is achieved, one may apply Lemma~\ref{lem: strong sep} to pass subcollections of $\cB_X'$ and $\cB_Y'$ which satisfy 1) in addition to 2) and 3), with a larger (but nevertheless still admissible) choice of $R$. More precisely, we first apply Lemma~\ref{lem: strong sep} to, say, the collection $\cB_X'$ (or balls obtained by suitably dilating the $B_X' \in \cB_X'$) to pass to a subcollection which  satisfies the strong separation. One can then pass to a suitable subcollection of $\cB_Y'$ by using the comparability property 3).\medskip

We now turn to the task of constructing the sets $\cB_X'$ and $\cB_Y'$ satisfying 2) and 3). To this end, we will construct a sequence of balls $B_X^1, \dots, B_X^{\ell} \in \cB_X$ and $B_Y^1, \dots, B_Y^{\ell} \in \cB_Y$ and a sequence of constants $C_{\ell} \geq \cdots \geq C_1 \geq 1$ using a recursive algorithm. In particular, defining
\begin{align}
\label{eq: Vitali -1}
    \cB_{X,\ell} &:= \{B_X \in \cB_X : B_X  \cap C_k \cdot B_X^k = \emptyset \textrm{ for all $1 \leq k \leq \ell$}\}, \\
\nonumber
    \cB_{Y,\ell} &:= \{B_Y \in \cB_Y : B_Y  \cap C_k \cdot B_Y^k = \emptyset \textrm{ for all $1 \leq k \leq \ell$}\},
\end{align}
these objects have the following properties:\medskip

\noindent 1) \textbf{Strong separation.} Let $\rho \geq 1$ be a fixed parameter, chosen suitably large depending only on $n$ and $\lambda$ to satisfy the forthcoming requirements of the proof. 
\begin{enumerate}
\item[1)$_{X, \ell}$] If $B_X \in \cB_{X,\ell}$, then $B_X \cap \rho C_k \cdot B_{X}^k = \emptyset$ for $1 \leq k \leq \ell$, \smallskip
\item[1)$_{\,Y, \ell}$] If $B_Y \in \cB_{Y,\ell}$, then $B_Y \cap \rho C_k \cdot B_{Y}^k = \emptyset$ for $1 \leq k \leq \ell$.
\end{enumerate}
This condition will play a minor technical role in the proof.\smallskip

\noindent 2) \textbf{Vitali condition.} Let $C = C(n, \lambda) := \lambda n^{1/d}$.
\begin{enumerate}
    \item[2)$_{X, \ell}$] If $B_X \in \cB_{X,k-1}$, then $\rad B_X \leq C\rad B_X^k$ for $1 \leq k \leq \ell$; \smallskip
    \item[2)$_{\,Y, \ell}$]  If $B_Y \in \cB_{Y,k-1}$, then $\rad B_Y \leq C \rad B_Y^k$ for $1 \leq k \leq \ell$;
\end{enumerate}
\noindent 3) \textbf{Comparable balls.} Let $\bar{C} = \bar{C}(n,\lambda) := 2C +1$. 
\begin{enumerate}
    \item[3)$_{\ell}$] For all $1 \leq k \leq \ell$, we have
    \begin{equation*}
         B_X^k \subseteq \bar{C} \cdot B_Y^k \quad \textrm{and} \quad  B_Y^k \leq \bar{C} \cdot B_X^k.
    \end{equation*}
\end{enumerate}

Suppose $B_X^1, \dots, B_X^{\ell} \subseteq \cB_X$, $B_Y^1, \dots, B_Y^{\ell} \subseteq \cB_Y$ and $(C_k)_{k=1}^{\ell}$ have already been constructed and satisfy the properties listed above. \medskip

\noindent \texttt{Stopping condition.} If either $\cB_{X,\ell} = \emptyset$ or $\cB_{Y,\ell} = \emptyset$, then the algorithm terminates.\medskip

\noindent \texttt{Recursive step.}  Suppose the algorithm has not terminated at step $\ell$ so that $\cB_{X,\ell} \neq \emptyset$ and $\cB_{Y,\ell} \neq \emptyset$. Let $B_X^{\ell+1, *} \in \cB_{X,\ell}$ and $B_Y^{\ell+1, *} \in \cB_{Y,\ell}$ be balls of maximal radii lying in these sets. 

By symmetry, we may assume that $\rad B_X^{\ell+1,*} \geq \rad B_Y^{\ell+1,*}$. In this case, we define $B_X^{\ell+1} := B_X^{\ell+1,*}$, so that Property 2)$_{X, \ell+1}$ clearly holds.

We claim that
\begin{equation}\label{eq: Vitali 1}
   B_X^{\ell+1} \subseteq \bigcup_{B_Y \in \cB_{Y,\ell}} \lambda \cdot B_Y.
\end{equation}
Indeed, suppose the above inclusion fails so that, by the hypothesis \eqref{eq: Vitali lem}, there exists some $B_Y \in \cB_Y \setminus \cB_{Y,\ell}$ such that
\begin{equation}\label{eq: Vitali 2}
   B_X^{\ell+1} \cap \lambda \cdot B_Y  \neq \emptyset.  
\end{equation}
Since $B_Y \in \cB_Y \setminus \cB_{Y,\ell}$, there exists some $1 \leq k \leq \ell$ such that 
\begin{equation}\label{eq: Vitali 2.5}
    B_Y \cap C_k \cdot B_Y^k \neq \emptyset. 
\end{equation}
We choose $k$ to be minimal with this property. Thus, $B_Y \cap C_j \cdot B_Y^j = \emptyset$ for all $1 \leq j \leq k-1$, which is precisely the condition $B_Y \in \cB_{Y,k-1}$. Consequently, by Property 2)$_{\,Y, k}$, we have $\rad B_Y \leq C \rad B_Y^k$. Recalling \eqref{eq: Vitali 2.5} and applying the triangle inequality in the form of \eqref{eq: triangle ineq} together with Property 3)$_{\ell}$, we see that
\begin{equation}\label{eq: Vitali 3}
    \lambda \cdot B_Y \subseteq (\rho/2)C_k \cdot B_Y^k \subseteq \rho C_k \cdot B_X^k,
\end{equation}
provided $\rho$ is suitably chosen. Combining \eqref{eq: Vitali 2} and \eqref{eq: Vitali 3}, we have 
\begin{equation*}
    B_X^{\ell+1} \cap \rho C_k \cdot B_X^k \neq \emptyset.
\end{equation*} 
By Property 1)$_{\,X, \ell}$, it follows that $B_X^{\ell+1} \notin \cB_{X,\ell}$, but this contradicts our choice of $B_X^{\ell+1}$. \smallskip

In view of \eqref{eq: Vitali 1}, we fix some $B_Y^{\ell+1} \in \cB_{Y,\ell}$ such that $\lambda \cdot B_Y^{\ell+1}$ has non-trivial intersection with $B_X^{\ell+1}$ with maximal possible radius. It follows that 
\begin{equation}\label{eq: Vitali 4.5}
    \rad B_Y^{\ell+1} \leq \rad B_Y^{\ell+1,*} \leq \rad B_X^{\ell+1,*} = \rad B_X^{\ell+1},
\end{equation} 
whilst
\begin{equation*}
    \mathcal{L}^d\big(B_X^{\ell+1}\big) \leq \sum_{\substack{B_Y \in \cB_{Y,\ell} \\ \lambda \cdot B_Y \cap B_X^{\ell+1} \neq \emptyset}} \mathcal{L}^d\big(\lambda \cdot B_Y\big) \leq C^d \mathcal{L}^d\big(B_Y^{\ell+1}\big),
\end{equation*}
where $\mathcal{L}^d$ denotes the $d$-dimensional Lebesgue measure. Thus, $\rad B_X^{\ell+1} \leq C B_Y^{\ell+1}$. 
Combining these observations with \eqref{eq: triangle ineq} establishes Property 3)$_{\ell+1}$. Similarly, arguing as in \eqref{eq: Vitali 4.5}, given $B_Y \in \cB_{Y,\ell}$, it follows that $\rad B_Y \leq \rad B_X^{\ell+1}$. Combining this with Property 3)$_{\ell+1}$, we have $\rad B_Y  \leq C \rad B_Y^{\ell+1}$, and so Property 2)$_{\,Y,\ell+1}$ also holds.\smallskip

It remains to construct the constant $C_{\ell+1}$ and verify 1)$_{X,\ell+1}$ and 1)$_{\,Y,\ell+1}$. For $0 \leq m \leq 2n+1$ consider the sets
\begin{align*}
    \cB_{X,\ell+1}^{(m)} &:= \cB_{X,\ell} \cap \{B_X \in \cB_X : B_X \cap C_{\ell} \rho^m \cdot B_X^{\ell+1} = \emptyset \}, \\
    \cB_{Y,\ell+1}^{(m)} &:= \cB_{Y,\ell} \cap  \{B_Y \in \cB_Y : B_Y \cap C_{\ell} \rho^m \cdot B_Y^{\ell+1} = \emptyset\},
\end{align*}
By a pigeonholing argument, there must exist some choice of $0 \leq m \leq 2n$ such that
\begin{equation}\label{eq: stability}
    \cB_{X,\ell+1}^{(m+1)}  = \cB_{X,\ell+1}^{(m)}  \quad \textrm{and} \quad \cB_{Y,\ell+1}^{(m+1)}  = \cB_{Y,\ell+1}^{(m)} .
\end{equation}
With this fixed value of $m$, define $C_{\ell+1} := C_{\ell} \rho^m$ so that $\cB_{X,\ell+1} = \cB_{X,\ell+1}^{(m)}$ and $\cB_{Y,\ell+1} = \cB_{Y,\ell+1}^{(m)}$. It immediately follows from \eqref{eq: stability} and 1)$_{X,\ell}$ and 1)$_{\,Y,\ell}$ that 1)$_{X,\ell+1}$ and 1)$_{\,Y,\ell+1}$ hold.

\medskip

The above algorithm must terminate after finitely many steps since the $\cB_{X,\ell}$ as defined in \eqref{eq: Vitali -1} form nested sequence of subsets of the finite set $\cB_X$ of strictly decreasing cardinality. Indeed, note that $B_X^{\ell+1}$ is chosen from $\cB_{X,\ell}$ in the above algorithm, so that $B_X^{\ell+1} \in \cB_{X,\ell}$ whilst clearly $B_X^{\ell+1} \notin \cB_{X, \ell+1}$. Suppose the algorithm terminates after the $L$th step. We show that the resulting families $\cB_X' := \{B_X^1, \dots, B_X^L\}$ and $\cB_Y' := \{B_Y^1, \dots, B_Y^L\}$ satisfy the desired properties 2) and 3) from the statement of the lemma.  \smallskip

First we note that, provided $R \geq \bar{C}$, Property 3) immediately follows from Property 3)$_L$ of the algorithm.\smallskip

It remains to show Property 2). By the definition of the stopping condition, we know either $\cB_{X,L} = \emptyset$ or $\cB_{Y, L} = \emptyset$. By symmetry we may assume that $\cB_{X,L} = \emptyset$. Using the standard Vitali covering argument, Property 2)$_{X,L}$ implies that 
\begin{equation}\label{eq: strong Vitali}
    \textrm{for all $B_X \in \cB_X$ there exists some $1 \leq \ell \leq L$ such that $B_X \subseteq (R/8) \cdot B_X^{\ell}$,}
\end{equation}
 provided $R \geq 1$ is chosen sufficiently large depending only on $n$ and $\lambda$. This is a slightly stronger version of the first inclusion in Property 2) of the lemma. We turn to the second inclusion. If $B_Y \in \cB_Y \setminus \cB_{Y,L}$, then we may argue as above to show that $\cB_Y \subseteq (R/8) \cdot B_Y^{\ell}$ for some $1 \leq \ell \leq L$. Thus, it suffices to consider the case $B_Y \in \cB_{Y,L}$. By \eqref{eq: Vitali lem} and \eqref{eq: strong Vitali}, we know $B_Y \cap (R/4) \cdot B_X^{\ell} \neq \emptyset$ for some $1 \leq \ell \leq L$. On the other hand, Property 2)$_{Y,L}$ of the algorithm implies $\rad B_Y \leq C \rad B_Y^{\ell}$ for all $1 \leq \ell \leq L$.  Thus, \eqref{eq: triangle ineq} and Property 3)$_L$ give us
\begin{equation*}
B_Y \subseteq \bigcup_{\ell = 1}^L (R/2) \cdot B_X^{\ell} \subseteq \bigcup_{\ell = 1}^L R \cdot B_Y^{\ell},
\end{equation*}
again provided $R$ is chosen sufficiently large. 
\end{proof}




\bibliography{Reference}{}
\bibliographystyle{plain}

\end{document}